\documentclass[leqno]{article}

\usepackage{layout}
\usepackage[final]{pdfpages}

\usepackage[OT1]{fontenc}
\usepackage{color}
\usepackage{amsthm,amsmath,graphicx,latexsym,amssymb,amscd,amsfonts,enumerate, bm}
\usepackage[colorlinks,citecolor=blue,linkcolor=red,urlcolor=blue]{hyperref}
\usepackage{authblk} 

\theoremstyle{plain}
\newtheorem{theorem}{Theorem}
\newtheorem{proposition}[theorem]{Proposition}
\newtheorem{lemma}[theorem]{Lemma}
\newtheorem{corollary}[theorem]{Corollary}

\newtheorem{remark}[theorem]{Remark}
\newtheorem{example}[theorem]{Example}

\newcommand{\dd}{\mathrm{d}}
\newcommand{\F}{\mathcal{F}}
\def\E{\mathbb{E}}
\def\R{\mathbb{R}}
\def\P{\mathbb{P}}
\def\ve{\varepsilon}

\begin{document}

\title{\textbf{\Large Strong approximation of stochastic differential equations driven by a time-changed Brownian motion with time-space-dependent coefficients}}

\date{ }

\author{Sixian Jin\thanks{Fordham University. Email:\ sjin27@fordham.edu} \ and Kei Kobayashi\thanks{Corresponding author. Fordham University. Email:\ kkobayashi5@fordham.edu}}

\maketitle

%
%

\begin{abstract}
The rate of strong convergence is investigated for an approximation scheme for a class of stochastic differential equations driven by a time-changed Brownian motion, where the random time changes $(E_t)_{t\ge 0}$ considered include the inverses of stable and tempered stable subordinators as well as their mixtures. Unlike those in the work of Jum and Kobayashi (2016), the coefficients of the stochastic differential equations discussed in this paper 
depend on the regular time variable $t$ rather than the time change $E_t$. 
This alteration makes it difficult to apply the method used in that paper.
To overcome this difficulty, we utilize a Gronwall-type inequality involving a stochastic driver  
to control the moment of the error process.
Moreover, in order to guarantee that an ultimately derived error bound is finite, we establish a useful criterion for the existence of exponential moments of powers of the random time change. \vspace{3mm}

\noindent\textit{Key words:} stochastic differential equation, numerical approximation, rate of convergence, inverse subordinator, random time change, time-changed Brownian motion.\vspace{3mm}

\noindent\textit{2010 Mathematics Subject Classification:} 65C30, 60H10.

\end{abstract}

\section{Introduction}\label{section_introduction}

Let $B=(B_t)_{t\ge0}$ be a standard Brownian motion and $E=(E_t)_{t\ge0}$ be a stochastic process defined by the inverse of a subordinator $D=(D_t)_{t\ge0}$ with infinite L\'evy measure, independent of $B$. The composition $B\circ E=(B_{E_t})_{t\ge0}$, called a \textit{time-changed Brownian motion}, and its various generalizations have been widely used to model anomalous diffusions arising in e.g.\ physics \cite{MetzlerKlafter00,Nigmatullin,Zaslavsky}, finance \cite{GMSR,MagdziarzOrzelWeron}, hydrology \cite{BWM}, and cell biology \cite{Saxton}. 
See Chapter 1 of \cite{HKU-book} for details.
The time-changed Brownian motion is non-Markovian (\cite{MS_1,MS_2}). 
Also, the identity $\E[B_{E_t}^2]=\E[E_t]$ holds, 
and in particular, if the subordinator $D$ is stable with index $\beta\in(0,1)$,
then $\E[B_{E_t}^2]=t^\beta/\Gamma(1+\beta)$,
which shows that in large time scales
particles represented by $B\circ E$ spread at a slower rate than the rate at which Brownian particles diffuse. 
 Moreover, the densities of $B\circ E$ satisfy the time-fractional order Fokker--Planck (or forward Kolmogorov) equation $\partial_t^\beta u(t,x) =(1/2)\Delta u(t,x)$, where $\partial_t^\beta$ denotes the Caputo fractional derivative of order $\beta$ with respect to the variable $t$. 
Various extensions of $B\circ E$ and their associated 
fractional order
 Fokker--Planck equations have been investigated, including time-changed fractional Brownian motions (see \cite{HKRU,HKU-2,MNX}) and stochastic differential equations (SDEs) involving the random time change $E$ (see below).

 In this paper, we investigate the rate of strong convergence of a numerical approximation scheme for an SDE of the form 
\begin{align}\label{SDE_001_0}	
	X_t=x_0+\int_0^t F(s,X_s)\, \dd E_s+\int_0^t G(s,X_s)\,\dd B_{E_s}, 
\end{align}
where the coefficients $F$ and $G$ satisfy some regularity conditions. 
Note that since $E$ is a nondecreasing process and $B\circ E$ is a martingale with respect to a certain filtration, 
SDE \eqref{SDE_001_0} is understood within the framework of stochastic integrals driven by semimartingales; see \cite{Kobayashi} or the beginning of Section \ref{section_approximation} for details.
SDE \eqref{SDE_001_0} and its extensions to cases involving jump components have recently drawn more and more attention. For example, 
papers \cite{NaneNi2017,NaneNi2018,Wu}
established stability in various senses of solutions of SDEs driven by time-changed processes using the time-changed It\^o formula derived in \cite{Kobayashi} and its generalizations. 
In \cite{MagdziarzZorawik,NaneNi2016}, fractional order Fokker--Planck equations were derived for solutions of SDEs of the form \eqref{SDE_001_0} with L\'evy noise terms added.
Practical situations where such SDEs naturally arise include Langevin-type subdiffusive dynamics in physics with force terms of the form $F(t,x)=f_1(t)f_2(x)$ and constant diffusion coefficient \cite{Heinsalu2007,WeronMagdziarzWeron2008} and stock price dynamics with trapping events and volatility clustering in finance where the diffusion coefficient may also take the form $G(t,x)=g_1(t)g_2(x)$ \cite{Lv2012,Magdziarz_comment,Onalan2015}.
 As the well-established theory of classical It\^o SDEs (without a random time change) enabled a number of mathematicians and scientists in various fields to explore questions about more complicated diffusion processes than the Brownian motion itself, further investigations of SDE \eqref{SDE_001_0} and its extensions are necessary and expected in order to deal with more complex anomalous diffusions.

This paper partly builds upon some results established in \cite{Magdziarz_simulation,Magdziarz_spa}, which investigated the process $X$ defined by \eqref{SDE_001_0} with $F(t,x)=F(t)$ having finite variation and $G(t,x)\equiv 1$ (in which case \eqref{SDE_001_0} is no longer an SDE since the coefficients do not depend on $x$). In those papers, a numerical approximation scheme for $X$ was presented together with the rate of strong convergence. 
On the other hand, even though the idea employed in \cite{Magdziarz_simulation,Magdziarz_spa} together with the Euler--Maruyama scheme allows us to approximate the solution $X$ of SDE \eqref{SDE_001_0} with general space-time-dependent coefficients, convergence of the approximation scheme has not been investigated. Establishing the rate of convergence for the scheme is an extremely important issue both theoretically and practically in numerical analysis of complex systems displaying anomalous dynamics, and that is the main contribution of this paper. 
In particular, our convergence results will help justify the use of Monte Carlo techniques in approximating the solutions of the fractional order Fokker--Planck equations derived in \cite{MagdziarzZorawik,NaneNi2016}.

The main difficulty in analyzing SDE \eqref{SDE_001_0} lies in the 
``asynchrony'' between the time variable $s$ in the integrands and the time change $E_s$ in the driving processes. To see this, consider instead an SDE with a synchronized time clock of the form 
\begin{align}\label{SDE_002_0}	
	X_t=x_0+\int_0^t F(E_s,X_s)\, \dd E_s+\int_0^t G(E_s,X_s)\, \dd B_{E_s},
\end{align}
where the coefficients depend on $E_s$ rather than $s$. For this SDE, 
the associated fractional order Fokker--Planck equation was established in \cite{HKU-1}, and the orders of strong and weak convergence of an approximation scheme were derived in \cite{JumKobayashi}.
 The key fact used in those papers was the \textit{duality principle} between SDE \eqref{SDE_002_0} and the classical It\^o SDE 
\begin{align}\label{SDE_003_0}	
	Y_t=x_0+\int_0^t F(s,Y_s)\, \dd s+\int_0^t G(s,Y_s)\, \dd B_s. 
\end{align}
Namely, if $Y_t$ solves \eqref{SDE_003_0}, then $X_t:=Y_{E_t}$ solves \eqref{SDE_002_0}, while if $X_t$ solves \eqref{SDE_002_0}, then $Y_t:=X_{D_t}$ solves \eqref{SDE_003_0}, where $D$ is the original subordinator (see Theorem 4.2 of \cite{Kobayashi}). This one-to-one correspondence between the two SDEs allows us to approximate the solution of \eqref{SDE_002_0} by the composition 
	$X^\delta_t:=(Y^\delta\circ E^\delta)_t=Y^\delta_{E^\delta_t}$,
 where $\delta\in(0,1)$ refers to an equidistant step size, $E^\delta$ is the approximation process for $E$ defined in \cite{Magdziarz_simulation,Magdziarz_spa}, and $Y^\delta$ is the approximation of $Y$ based on the Euler--Maruyama scheme. The independence assumption between $B$ and $E$ together with representation \eqref{SDE_003_0} implies independence between $Y$ and $E$, and therefore, the two approximation processes $Y^\delta$ and $E^\delta$ can be constructed independently and simply composed 
 to define the approximation process $X^\delta$. 
 The independence also allows the two types of errors (one ascribed to the approximation of $Y$ and the other due to the approximation of $E$) to be analyzed separately.

On the other hand, a referee of the paper \cite{JumKobayashi} raised an important question of whether the methods used for SDE \eqref{SDE_002_0} can be applied to SDE \eqref{SDE_001_0} or not. 
Unfortunately, \textit{the approach used in \cite{JumKobayashi} no longer works for approximation of the solution of SDE \eqref{SDE_001_0}.} Indeed, the duality principle implies the corresponding SDE takes the form 
\begin{align}\label{SDE_004_0}	
	Y_t=x_0+\int_0^t F(D_{s-},Y_s)\, \dd s+\int_0^t G(D_{s-},Y_s)\, \dd B_s,
\end{align}
which clearly shows $Y$ depends on $D$ (and hence on $E$ as well), and consequently, the conditioning argument based on the independence of $Y$ and $E$ used for SDE \eqref{SDE_002_0} cannot be applied. This observation, which appears in Remark 3.2(5) of \cite{JumKobayashi},  
forces us to take a 
different
 approach in dealing with SDE \eqref{SDE_001_0}. In particular, the duality principle is not used at all. 
Instead, we utilize a Gronwall-type inequality involving a stochastic driver  
to control the moment of the error process.
Moreover, in order to eventually obtain a meaningful bound for the moment in Section 3,
we derive a useful criterion for the existence of the exponential moment $\E[e^{\lambda E_t^r}]$ of the $r$th power of the inverse subordinator in Section 2; 
this may be of independent interest to some readers.
It is also worth mentioning that, even though the approximation scheme used in this paper is of Euler--Maruyama type, the order of strong uniform convergence to be established in Theorem \ref{Theorem_with_noise} is strictly less than $1/2$. This is different from the classical setting of It\^o SDEs without a random time change, where the order $1/2$ can be achieved. An explanation of why this phenomenon occurs under the time change will be discussed in 
Remark \ref{Remark_with_noise}(3).

\section{Exponential moments of powers of inverse subordinators}\label{section_inverse}

Throughout the paper, 
$(\Omega,\F,\P)$ denotes a complete probability space and
 $D=(D_t)_{t\ge 0}$ denotes a subordinator starting at $0$ with Laplace exponent $\psi$ with killing rate 0, drift 0, and L\'evy measure $\nu$; i.e.\ $D$ is a one-dimensional nondecreasing L\'evy process with c\`adl\`ag paths starting at 0 with Laplace transform 
\begin{align}\label{def_LaplaceExponent}
	\mathbb{E}[e^{-sD_t}]=e^{-t\psi(s)}, \ \ \textrm{where} \ \ 
	\psi(s)=\int_0^\infty (1-e^{-sy})\,\nu(\dd y), \ \ s> 0,
\end{align}
with 
$\int_0^\infty (y\wedge 1)\, \nu(\dd y)<\infty$. 
\textit{We focus on the case when the L\'evy measure $\nu$ is infinite} (i.e.\ $\nu(0,\infty)=\infty$),
 which implies compound Poisson subordinators are excluded from our discussion.
Let $E=(E_t)_{t\ge 0}$ be the inverse of $D$; i.e.\ 
\[
	E_t:=\inf\{u>0; D_u>t\}, \ \ t\ge 0. 
\]
We call $E$ an \textit{inverse subordinator}. 
The assumption that $\nu(0,\infty)=\infty$ implies that $D$ has strictly increasing paths with infinitely many jumps (see e.g.\ \cite{Sato}), and therefore, $E$ has continuous, nondecreasing paths starting at 0. If the subordinator $D$ is stable with index $\beta\in(0,1)$, then $\psi(s)=s^\beta$ and the corresponding time change $E$ is called an \textit{inverse $\beta$-stable subordinator.}   
Note that the jumps of $D$ correspond to the (random) time intervals on which $E$ is constant, and during those constant periods, any time-changed process of the form  $X\circ E=(X_{E_t})_{t\ge 0}$ also remains constant. If $B$ is a standard Brownian motion independent of $D$, we can regard particles represented by the \textit{time-changed Brownian motion} $B\circ E$ as being trapped and immobile during the constant periods. See Figure \ref{figure_001}.

\begin{figure}
    \centering
    \includegraphics[width=3.8in]{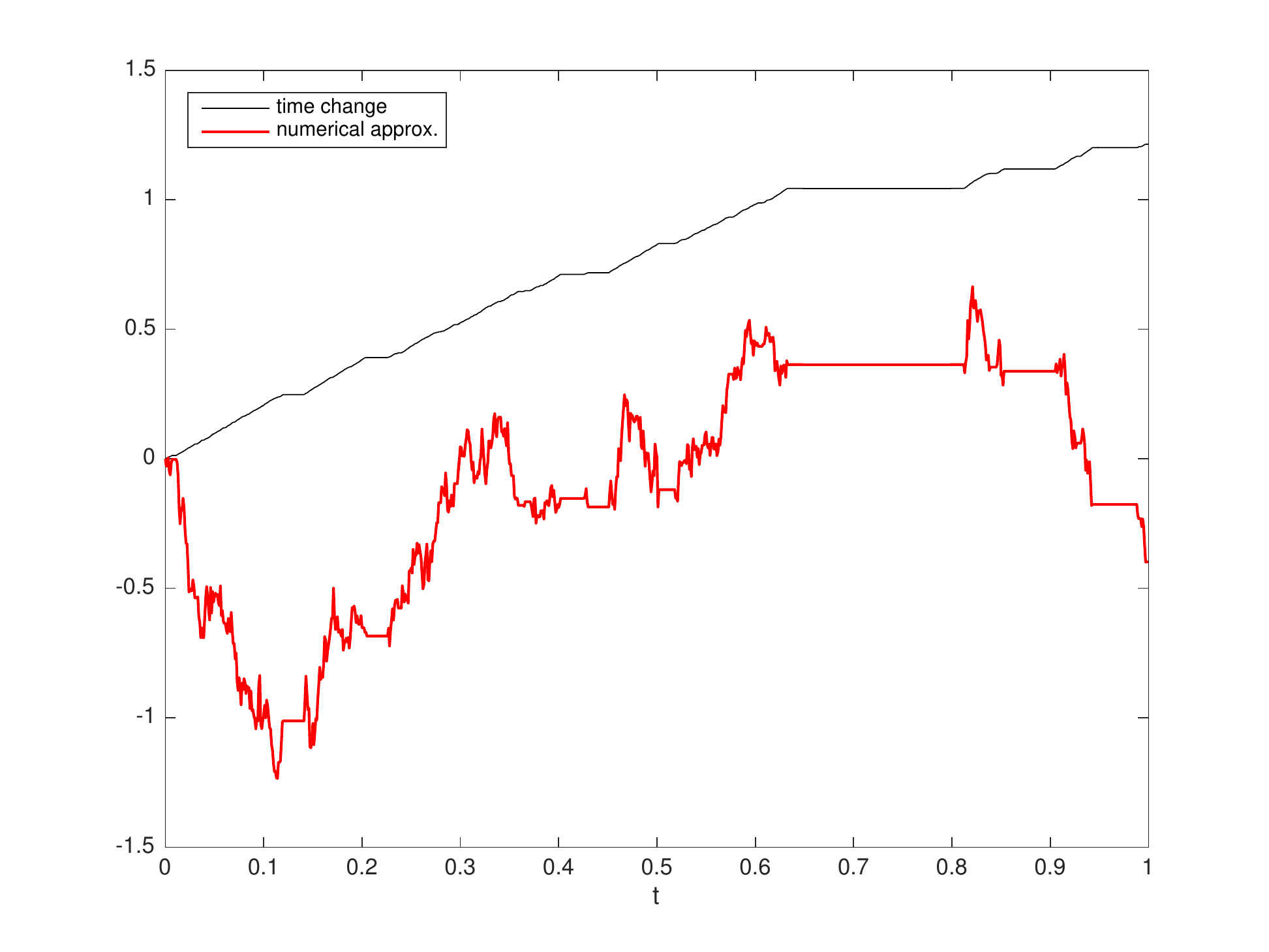}
    \caption{Sample paths of an inverse $0.9$-stable subordinator $E$ (black) and the corresponding time-changed Brownian motion $B\circ E$ (red), which share the same constant periods.}
    \label{figure_001}
\end{figure}

Any inverse subordinator $E$ with infinite L\'evy measure is known to have the exponential moment; i.e.\ $\mathbb{E}[e^{\lambda E_t}]<\infty$ for all $\lambda>0$ and $t> 0$ (see \cite{JumKobayashi,MagdziarzOrzelWeron}).
However, as is shown in Theorem \ref{Theorem_general}, whether the expectation $\mathbb{E}[e^{\lambda E_t^r}]$ with $r>1$ exists or not depends on the nature of the time change. 
In particular, if $E$ is an inverse $\beta$-stable subordinator, then 
 $\mathbb{E}[e^{\lambda E_t^2}]$ exists if $1/2<\beta<1$ while it does not if $0<\beta<1/2$. When $\beta=1/2$, whether the expectation exists or not depends on the relationship between $\lambda$ and $t$; 
 see Remark \ref{Remark_general}(2).

One situation where the need for the information about the existence  of $\mathbb{E}[e^{\lambda E_t^r}]$ arises is implicitly discussed in \cite{JumKobayashi}. 
Namely, consider a sequence $\{X^{(n)}\}_{n\ge 1}$ of stochastic processes 
converging to $X$ in $L^p$ with $p\ge 1$ uniformly on compact subsets of $[0,\infty)$ with a bound 
$
	\|\sup_{s\in[0,t]}|X^{(n)}_s-X_s|\|_{L^p(\Omega)} 
	\le a_n g(t)
$
for all $t>0$, 
where $\{a_n\}$ is a sequence approaching 0
and $g(t)$ is a function of $t$.
If $E$ is an inverse subordinator independent of both $X$ and $\{X^{(n)}\}$, then a simple conditioning argument yields
\begin{align*}
	\biggl\|\sup_{s\in[0,t]}|X^{(n)}_{E_s}-X_{E_s}|\biggr\|_{L^p(\Omega)} 
	=\biggl\|\sup_{s\in[0,E_t]}|X^{(n)}_{s}-X_{s}|\biggr\|_{L^p(\Omega)} 
	\le a_n \left\|g(E_t)\right\|_{L^p(\Omega)}.
\end{align*}
Therefore, if e.g.\ $g(t)$ takes the form $g(t)=ce^{ct^2}$ and $E$ is an inverse $\beta$-stable subordinator with $\beta\in(0,1/2)$, then by the discussion in the previous paragraph,
$\left\|g(E_t)\right\|_{L^p(\Omega)}=c\hspace{1pt}(\mathbb{E}[e^{pcE_t^2}])^{1/p}=\infty$,
 and consequently, the above bound is no longer meaningful. 
Such a simple example illustrates the significance of criteria for the existence and non-existence of the expectation of the form $\mathbb{E}[e^{\lambda E_t^r}]$.

To describe the kinds of inverse subordinators we are mainly concerned with in this paper, let us introduce the notion of regularly varying and slowly varying functions. 
A function $f:(0,\infty)\to (0,\infty)$ is said to be \textit{regularly varying at $\infty$ with index $\alpha\in\mathbb{R}$} if 
$
	\lim_{s\to \infty} f(cs)/f(s)=c^\alpha
$
for any $c>0$. 
We denote by $\mathrm{RV}_\alpha$ the class of regularly varying functions at $\infty$ with index $\alpha$.
A function $\ell:(0,\infty)\to (0,\infty)$ is said to be \textit{slowly varying at $\infty$} if $\ell\in \mathrm{RV}_0$ 
(i.e.\ $\ell\in \mathrm{RV}_\alpha$ with $\alpha=0$).
Every regularly varying function $f$ with index $\alpha\in\mathbb{R}$ is represented as 
$
	f(s)=s^\alpha \ell(s)
$
 with $\ell$ being a slowly varying function.

Note that the following two Laplace exponents are regularly varying at $\infty$ with index $\beta\in(0,1)$: $\psi(s)=s^\beta$, which corresponds to a stable subordinator with index $\beta$, and $\psi(s)=(s+\kappa)^\beta-\kappa^\beta$ with $\kappa>0$, which corresponds to an exponentially tempered (or tilted) stable subordinator with index $\beta$ and tempering factor $\kappa$. 
On the other hand, $\psi(s)=\log(1+s)$, which corresponds to a Gamma subordinator, is slowly varying at $\infty$.
We now state the main theorem of this section, which will be used in the proof of Theorem \ref{Theorem_with_noise} in the next section.

\begin{theorem}\label{Theorem_general}
Let $E$ be the inverse of a subordinator $D$ whose Laplace exponent $\psi$ is regularly varying at $\infty$ with index $\beta\in[0,1)$. 
If $\beta=0$, assume further that $\nu(0,\infty)=\infty$. Fix $\lambda>0$, $t> 0$ and $r>0$.  
\begin{enumerate}[(1)]
\item If 
$r<1/(1-\beta)$, then $\mathbb{E}[e^{\lambda E^r_t}]<\infty$.
\item If 
$r>1/(1-\beta)$, then $\mathbb{E}[e^{\lambda E^r_t}]=\infty$. 
\end{enumerate}
\end{theorem}

To establish Theorem \ref{Theorem_general}, we will use tail probability estimates for subordinators given in \cite{JainPruitt}. The estimates are given in terms of the Laplace exponent $\psi$ and hence easily applicable to quite general situations.
Let us introduce some auxiliary notations used in Section 5 of \cite{JainPruitt}. For a subordinator $D$ with Laplace exponent $\psi$ in \eqref{def_LaplaceExponent} and infinite L\'evy measure (i.e.\ $\nu(0,\infty)=\infty$), let 
\begin{align}\label{def_g(s)}
	g(s):=\psi'(s), \ \  R(s):=\psi(s)-s\psi'(s), \ \ s> 0.
\end{align}
Note that $\psi$ is a Bernstein function defined on $(0,\infty)$; i.e.\ $\psi\in C^\infty(0,\infty)$ with $\psi> 0$ and $(-1)^n\psi^{(n)}< 0$ for all $n\in\mathbb{N}$. This particularly implies that $g$ is a continuous, strictly decreasing function with $g(0)=\int_0^\infty y \,\nu(\dd y)$ and $g(\infty)=0$, while $R$ is a continuous, strictly increasing function with $R(0)=0$ and $R(\infty)=\nu(0,\infty)=\infty$, 
where we follow the convention $f(0):=\lim_{s\to 0}f(s)$ and $f(\infty):=\lim_{s\to \infty}f(s)$ for a given function $f$ defined on $(0,\infty)$.
The condition $\nu(0,\infty)=\infty$ guarantees that the inverse $E$ of $D$ has a finite exponential moment; i.e.\ $\mathbb{E}[e^{\lambda E_t}]<\infty$ for all $t> 0$ and $\lambda>0$ (see \cite{JumKobayashi,MagdziarzOrzelWeron}). 

\begin{proposition}\label{Proposition_tail}
Let $E$ be the inverse of a subordinator $D$ with Laplace exponent $\psi$ 
and infinite L\'evy measure $\nu$ 
in \eqref{def_LaplaceExponent}.
Let $g(s)$ and $R(s)$ be defined as in \eqref{def_g(s)}. Fix $\lambda>0$, $t>0$ and $r>0$. 
\begin{enumerate}[(1)]
\item If there exist a constant $\ve>0$ and a function 
$x(s):[M,\infty)\to (g(\infty), g(0))$ with $M>0$ such that 
$sx(s)>t$ and $R(g^{-1}(x(s)))/s^{r-1}\ge \lambda+\ve$ for all $s\ge M$,
then 
$\mathbb{E}[e^{\lambda E^r_t}]<\infty$.
\item If there exist a constant $\ve>0$ and a decreasing function 
$x(s):[M,\infty)\to (g(\infty), g(0))$ with $M>0$ such that 
$sx(s)<t$ and $R(g^{-1}(x(s)))/s^{r-1}\le \lambda-\ve$ for all $s\ge M$,
then 
$\mathbb{E}[e^{\lambda E^r_t}]=\infty$.
\end{enumerate}
\end{proposition}
\begin{proof}
(1) By the assumption and  Lemma 5.2(i) of \cite{JainPruitt}, 
\begin{align}\label{proof_tail1}
	\mathbb{P}(D_s<t)
	\le \mathbb{P}(D_s\le sx(s))
	\le e^{-sR(g^{-1}(x(s)))}
	\le e^{-(\lambda+\ve)s^r}
\end{align}
for all $s\ge M$. Note that 
\begin{align*}
	\mathbb{E}[e^{\lambda E^r_t}]
	&=\mathbb{E}[(e^{E^r_t})^\lambda]
	=\int_0^{e^{M^r}} \lambda z^{\lambda-1} \mathbb{P}(e^{E^r_t}>z)\,\dd z
		+\int_{e^{M^r}}^\infty \lambda z^{\lambda-1} \mathbb{P}(e^{E^r_t}>z)\,\dd z.
\end{align*}
Since the first integral on the right hand side is finite, whether the expectation exists or not is completely determined by the second integral, which is estimated with the help of \eqref{proof_tail1} 
and the change of variables $z=e^{s^r}$ as 
\begin{align*}
	\int_{e^{M^r}}^\infty \hspace{-4pt}\lambda z^{\lambda-1} \mathbb{P}(e^{E^r_t}>z)\,\dd z
	&=\int_{M}^\infty \lambda r s^{r-1} e^{\lambda s^r} \mathbb{P}(D_s<t)\,\dd s
	\le \int_{M}^\infty \lambda r s^{r-1} e^{-\ve s^r}\,\dd s.
\end{align*}
Since the latter integral is finite, it follows that $\mathbb{E}[e^{\lambda E^r_t}]<\infty$.

(2) By the assumption and Lemma 5.2(ii) of \cite{JainPruitt}, 
there exists a constant $c>0$ such that for all $\eta>0$ and $s\ge M$, 
\begin{align*}
	\mathbb{P}(D_s<t)
	&\ge \mathbb{P}(D_s\le sx(s))
	\ge \biggl(1-\dfrac{(1+\eta)c}{\eta^2 sR(g^{-1}(x(s)))}\biggr) e^{-(1+2\eta)s R(g^{-1}(x(s)))}.
\end{align*}
With the choice of $\eta=\ve^2/(2(\lambda^2-\ve^2))$, we can find a constant $M_1\ge M$ with
\[
	\dfrac{(1+\eta) c}{\eta^2 s R(g^{-1}(x(M_1)))}\le \dfrac 12
\]
for all $s\ge M_1$ (since the fraction on the left hand side goes to 0 as $s\to \infty$).
The assumption that $x(s)$ is decreasing together with the fact that $g(s)$ is also decreasing implies $g^{-1}(x(s))$ and $R(g^{-1}(x(s)))$ are both increasing. 
Consequently, by the identity $1+2\eta=\lambda^2/(\lambda^2-\ve^2)$ and the above estimates, it follows that $\mathbb{P}(D_s<t)\ge (1/2) e^{-\lambda^2 s^r/(\lambda+\ve)}$ for $s\ge M_1$.
This, along with the change of variables $z=e^{s^r}$, gives a lower bound for $\int_{e^{M_1^r}}^\infty \lambda z^{\lambda-1} \mathbb{P}(e^{E^r_t}>z)\,\dd z$ 
as
\begin{align*}
	\int_{M_1}^\infty \lambda rs^{r-1} e^{\lambda s^r} \mathbb{P}(D_s<t)\,\dd s
	\ge \dfrac{\lambda}{2}rM_1^{r-1} \int_{M_1}^\infty e^{\lambda \ve s^r/(\lambda+\ve)} \,\dd s=\infty.
\end{align*}
This implies that $\mathbb{E}[e^{\lambda E^r_t}]=\infty$.
\end{proof}
\begin{corollary}\label{Corollary_tail}
Let $E$ be the inverse of a subordinator $D$ with Laplace exponent $\psi$
 and infinite L\'evy measure $\nu$ 
in \eqref{def_LaplaceExponent}.
Let $g(s)$ and $R(s)$ be defined as in \eqref{def_g(s)}. Fix $\lambda>0$, $t>0$ and $r>0$. 
\begin{enumerate}[(1)]
\item If there exists a function $x(s)$ defined for large $s$ and taking values in the interval $(g(\infty), g(0))$ such that
$sx(s)\to \infty$ and  
$R(g^{-1}(x(s)))/s^{r-1}\to \infty$ as $s\to\infty$,
then 
$\mathbb{E}[e^{\lambda E^r_t}]<\infty$. 
\item If there exists a decreasing function $x(s)$ defined for large $s$ and taking values in the interval $(g(\infty), g(0))$ such that 
$sx(s)\to 0$ and  
$R(g^{-1}(x(s)))/s^{r-1}\to 0$ as $s\to\infty$,
then 
$\mathbb{E}[e^{\lambda E^r_t}]=\infty$. 
\end{enumerate}
\end{corollary}
We now apply Corollary \ref{Corollary_tail} and the following lemma (in Propositions 1.5.1 and 1.5.7 of \cite{Bingham_book}) to prove Theorem \ref{Theorem_general}. 
\begin{lemma}\label{lemma_RV} 
\begin{enumerate}[(1)] 
\item Given $f\in \mathrm{RV}_\alpha$, $f(\infty)=\infty$ if $\alpha>0$, and $f(\infty)= 0$ if $\alpha<0$. 
\item If $f_i\in \mathrm{RV}_{\alpha_i}$ for $i=1,2$ and $f_2(\infty)=\infty$, then $f_1\circ f_2\in \mathrm{RV}_{\alpha_1\alpha_2}$.
\item If $f_i\in \mathrm{RV}_{\alpha_i}$ for $i=1,2$, then $f_1\cdot f_2\in \mathrm{RV}_{\alpha_1+\alpha_2}$. 
\end{enumerate}
\end{lemma}

\begin{proof}[Proof of Theorem \ref{Theorem_general}]
By Lemma \ref{lemma_RV}(1), $\psi(\infty)=\infty$ if $\beta\in(0,1)$. On the other hand, if $\beta=0$, by assumption, $\psi(\infty)=\nu(0,\infty)=\infty$. In any case, for any $c>0$, L'Hospital's rule gives
\[
	c^\beta
	=\lim_{s\to\infty} \frac{\psi(cs)}{\psi(s)}
	=c\lim_{s\to\infty} \frac{\psi'(cs)}{\psi'(s)}, \ \ \textrm{or} \ \ 
	\lim_{s\to\infty} \frac{\psi'(cs)}{\psi'(s)}=c^{\beta-1}, 
\]
so $\psi'\in \mathrm{RV}_{\beta-1}$. Moreover, since $\beta-1<0$, we have $\psi'(\infty)=0$ due to Lemma \ref{lemma_RV}(1), 
so it follows from L'Hospital's rule again that for any $c>0$, $\lim_{s\to\infty} \psi''(cs)/\psi''(s)=c^{\beta-2}$.
 Consequently, $-\psi''\in \mathrm{RV}_{\beta-2}$. (Note that $\psi''<0$ as $\psi$ is a Bernstein function.)

Letting $x(s):=g(s^r)=\psi'(s^r)$ and using Lemma \ref{lemma_RV}(2)(3) yields
\[
	sx(s)\in \mathrm{RV}_{(\beta-1)r+1},
\]
while
 $
 	R'(s)=\psi'(s)-(s\psi'(s))'=s(-\psi''(s))\in \mathrm{RV}_{\beta-1}.
$
Write $R'(s)$ as $s^{\beta-1}\ell(s)$ with a slowly varying function $\ell(s)$. If $\beta\in(0,1)$, then using Karamata's integral theorem (see Proposition 1.5.8 of \cite{Bingham_book}) yields 
\[
	R(s)=\int_0^s R'(r)\,\dd r=\int_0^s r^{\beta-1} \ell(r)\,\dd r \sim \ell(s)\int_0^s r^{\beta-1} \,\dd r= \frac{s^\beta \ell(s)}{\beta}
\]
as $s\to \infty$, so $R\in \mathrm{RV}_{\beta}$. 
On the other hand, if $\beta=0$, then $R\in \mathrm{RV}_0$ by Proposition 1.5.9a of \cite{Bingham_book}.
Thus, $R\in \mathrm{RV}_{\beta}$ regardless of the value of $\beta$. Hence,
\[
	\frac{R(g^{-1}(x(s)))}{s^{r-1}}=\frac{R(s^r)}{s^{r-1}}\in \mathrm{RV}_{\beta r -(r-1)}=\mathrm{RV}_{(\beta-1)r+1}.
\]
Again, by Lemma \ref{lemma_RV}(1),  
the two quantities $sx(s)$ and $R(g^{-1}(x(s)))/s^{r-1}$ both increase to $\infty$ if $(\beta-1)r+1>0$ and both decrease to $0$ if $(\beta-1)r+1<0$.  
Application of Corollary \ref{Corollary_tail} 
completes the proof. 
\end{proof}

\begin{example}\label{Example_RV}
\begin{em}
Fix $\lambda>0$ and $t>0$. 

(1) If $\psi(s)=(s+\kappa)^\beta-\kappa^\beta$ with $\beta\in(0,1)$ and $\kappa\ge 0$, then $\psi\in \mathrm{RV}_\beta$ regardless of the value of $\kappa$, so $\mathbb{E}[e^{\lambda E^r_t}]<\infty$ as long as $r<1/(1-\beta)$. This particularly implies that 
if the subordinator $D$ is stable or tempered stable with stability index $\beta\in(1/2,1)$, then $\mathbb{E}[e^{\lambda E^{2}_t}]<\infty$.
This fact will be used in the proof of Theorem \ref{Theorem_with_noise}. 

(2) If $\psi\in \mathrm{RV}_0$
 and $\psi(\infty)=\infty$ (e.g.\ $\psi(s)=\log(1+s)$, which corresponds to a Gamma subordinator $D$), then $\mathbb{E}[e^{\lambda E^r_t}]<\infty$ for any $0<r\le 1$, and $\mathbb{E}[e^{\lambda E^r_t}]=\infty$ for any $r>1$.

(3) The subordinator $D$ with Laplace exponent
	$\psi(s)=\int_0^1 s^\beta \,\rho(\dd \beta),$
where $\rho$ is a finite Borel measure on $(0,1)$ with $\mathrm{supp}(\rho)\subset (0,1)$, is regarded as a mixture of independent stable subordinators. 
 Its inverse $E$ can be used as a time change introducing more than one subdiffusive mode. In particular, 
Theorems 3.5 and 3.6 of \cite{HKU-1} establish that a class of SDEs driven by a time-changed L\'evy process with this particular time change is associated with a class of time-distributed fractional-order pseudo-differential equations. 
For specific applications where several subdiffusive modes appear, see e.g.\ \cite{GhoshWebb}. 
If $\rho=\sum_{j=1}^J a_j\delta_{\beta_j}$, where for each $j$, 
$a_j>0$
and $\delta_{\beta_j}$ is a Dirac measure with mass at $\beta_j\in(0,1)$, then 
since $\psi\in \mathrm{RV}_{\hat{\beta}}$ with $\hat{\beta}:=\max_{1\le j\le J}\beta_j$, it follows from Theorem \ref{Theorem_general} that $\mathbb{E}[e^{\lambda E^r_t}]<\infty$ for $0<r<1/(1-\hat{\beta})$, and $\mathbb{E}[e^{\lambda E^r_t}]=\infty$ for $r>1/(1-\hat{\beta})$. 
\end{em}
\end{example}

\begin{remark}\label{Remark_general}
\begin{em}
(1) Corollary \ref{Corollary_tail} can be possibly applied to subordinators whose Laplace exponents are regularly varying at $\infty$ with index $1$. For example, suppose $\widehat{\psi}(s):=as+\psi(s)$, 
where $a>0$ and $\psi\in \mathrm{RV}_\beta$
with $\beta\in[0,1)$ as in Theorem \ref{Theorem_general}. 
Then $\widehat{\psi}\in \mathrm{RV}_1$, $\widehat{\psi}'=a+\psi'\in \mathrm{RV}_0$, and $-\widehat{\psi}''=-\psi''\in \mathrm{RV}_{\beta-2}$. Consequently, letting $x(s):=g(s^r)$ yields $sx(s)\in \mathrm{RV}_1$ and $R(g^{-1}(x(s)))/s^{r-1}\in \mathrm{RV}_{(\beta-1)r+1}$, so both quantities go to $\infty$ as $s\to \infty$ if $(\beta-1)r+1>0$.
Thus, application of Corollary \ref{Corollary_tail} yields $\mathbb{E}[e^{\lambda E^r_t}]<\infty$  for  $r<1/(1-\beta)$. 
This implies that Theorem \ref{Theorem_with_noise} is applicable to subordinators with such Laplace exponents $\widehat{\psi}\in \mathrm{RV}_1$.

(2) If $D$ is stable with index $\beta\in(0,1)$, Theorem \ref{Theorem_general} can be obtained immediately from the known result about the moments of $E_t$ (see e.g.\ Corollary 3.1 of \cite{MS_1} and Proposition 5.6 of \cite{HKU-book}) together with the ratio test. Indeed, 
\begin{align*}
	\mathbb{E}[e^{\lambda E_t^r}]
	=\sum_{n=0}^\infty \dfrac{\lambda^n\mathbb{E}[E_t^{rn}]}{n!}
	=\sum_{n=0}^\infty\dfrac{\lambda^n}{n!}\dfrac{\Gamma(rn+1)}{\Gamma(rn\beta+1)}t^{rn\beta}
	=f(\lambda t^{r\beta}),
\end{align*}
where $f(z):=\sum_{n=0}^\infty a_n z^n$ with $a_n:=\Gamma(rn+1)/(n!\Gamma(rn\beta+1))$.
By Stirling's formula, as $n\to\infty$,
\begin{align*}
	\dfrac{a_{n+1}}{a_n}
	&=\frac{1}{n+1}\cdot \dfrac{\Gamma(rn+r+1)}{\Gamma(rn+1)}\cdot 
		\dfrac{\Gamma(rn\beta+1)}{\Gamma(r(n+1)\beta+1)}\\
	&\sim \frac{1}{n+1}\cdot \dfrac{(rn+r)^{rn+r+1/2} e^{-(rn+r)}}{(rn)^{rn+1/2}e^{-rn}}\cdot 
		\dfrac{(rn\beta)^{rn\beta+1/2}e^{-rn\beta}}
		{(r(n+1)\beta)^{r(n+1)\beta+1/2}e^{-r(n+1)\beta}}\\
	&=\frac{1}{n+1}\cdot \Bigl(\dfrac{n+1}{n}\Bigr)^{rn+1/2}\cdot \frac{(r(n+1))^{r}}{e^r}\cdot 
		\Bigl(\dfrac{n}{n+1}\Bigr)^{rn\beta+1/2}\cdot \dfrac{e^{r\beta}}{(r(n+1)\beta)^{r\beta}}\\
	&\sim \frac{1}{n+1}\cdot \dfrac{(r(n+1))^r}{(r(n+1)\beta)^{r\beta}}\\
	&\longrightarrow \begin{cases} 
						0 &\textrm{if} \  \ r<r\beta+1,\\ 
						\infty &\textrm{if} \ \ r>r\beta+1,\\
						r^r/(r-1)^{r-1} &\textrm{if} \ \ r=r\beta+1. 
					\end{cases}
\end{align*}
This yields Theorem \ref{Theorem_general}.
It also follows that in the threshold case when $r=1/(1-\beta)$, 
$\mathbb{E}[e^{\lambda E_t^r}]<\infty$ if $\lambda t^{r-1}<(r-1)^{r-1}/r^r$, while $\mathbb{E}[e^{\lambda E_t^r}]=\infty$ if $\lambda t^{r-1}>(r-1)^{r-1}/r^r$. This can also be verified using Proposition \ref{Proposition_tail}.
\end{em}
\end{remark}

\section{Approximation of SDEs with space-time-dependent coefficients}
\label{section_approximation}

Suppose the probability space $(\Omega,\F,\P)$ is equipped with a filtration $(\F_t)_{t\ge 0}$ satisfying the usual conditions. 
Let $B$ be an $m$-dimensional $(\F_t)$-adapted Brownian motion which is independent of an $(\F_t)$-adapted subordinator $D$ with infinite L\'evy measure.
Let $E$ be the inverse of $D$. 
Consider the SDE
\begin{align}\label{SDE_001}	
	X_t=x_0+\int_0^t F(s,X_s)\, \dd E_s+\int_0^t G(s,X_s)\,\dd B_{E_s} \ \ \textrm{for} \ t\in[0,T],
\end{align}
where $x_0\in\mathbb{R}^d$ is a non-random constant, $T>0$ is a fixed time horizon, and $F(t,x):[0,T]\times \mathbb{R}^d\to \mathbb{R}^d$ and $G(t,x):[0,T]\times \mathbb{R}^d\to \mathbb{R}^{d\times m}$ are measurable functions for which there exist constants $K>0$, ${\theta_F}\in(0,1]$ and $\theta_G\in(0,1]$ such that
\begin{align}
	\label{SDE_condition1} &|F(t,x)-F(t,y)|+|G(t,x)-G(t,y)|\le K|x-y|,\\
	\label{SDE_condition2} &|F(t,x)|+|G(t,x)|\le K(1+|x|),\\
	\label{SDE_condition3} &|F(s,x)-F(t,x)|\le K(1+|x|)|s-t|^{\theta_F},\\
	\label{SDE_condition4} &|G(s,x)-G(t,x)|\le K(1+|x|)|s-t|^{\theta_G}
\end{align}  
for all $x,y\in\mathbb{R}^d$ and $s,t\in [0,T]$, with $|\cdot|$ denoting the Euclidean norms of appropriate dimensions. \textit{In the remainder of the paper, we assume that $m=d=1$ for simplicity of discussions and expressions;} an extension to a multidimensional case is straightforward.  
For each fixed $t\ge 0$ the random time $E_t$ is an $(\F_t)$-stopping time, and therefore, the time-changed filtration $(\F_{E_t})_{t\ge 0}$ is well-defined. Moreover, since the time change $E$ is an $(\F_{E_t})$-adapted nondecreasing process and the time-changed Brownian motion $B\circ E=(B_{E_t})_{t\ge 0}$ is an $(\F_{E_t})$-martingale, 
SDE \eqref{SDE_001} is understood within the framework of stochastic integrals driven by semimartingales (see Corollary 10.12 of \cite{Jacod}; also see \cite{Kobayashi} for details).
Conditions \eqref{SDE_condition1}--\eqref{SDE_condition2} guarantee the existence of a unique strong solution of SDE \eqref{SDE_001} which is $(\F_{E_t})$-adapted. 
Conditions \eqref{SDE_condition3}--\eqref{SDE_condition4} are required to obtain strong convergence of our approximation scheme in Theorem \ref{Theorem_with_noise}. Note that in the classical setting of an It\^o SDE (i.e.\ SDE \eqref{SDE_001} with $E_t\equiv t$), the corresponding theorem for strong approximation usually assumes \eqref{SDE_condition3}--\eqref{SDE_condition4} with ${\theta_F}={\theta_G}=1/2$ (see Theorem 10.2.2 of \cite{KloedenPlaten}). Note also that we exclude cases when ${\theta_F}>1$ and/or ${\theta_G}>1$ since that would imply $F$ and/or $G$ must be independent of $t$ (i.e.\ $F(t,x)=F(x)$ and $G(t,x)=G(x)$).

As noted in Section \ref{section_introduction}, a standard conditioning approach used in \cite{JumKobayashi} based on the duality principle in \cite{Kobayashi} no longer works for SDE \eqref{SDE_001}. 
Our argument in this paper is different.
 We do not rely on the duality principle. 
Instead, we utilize a Gronwall-type inequality involving a stochastic driver to control the moment of the error process.
 Moreover, Theorem \ref{Theorem_general} established in Section \ref{section_inverse} will be used to guarantee that the error bound to be ultimately derived in the proof of Theorem \ref{Theorem_with_noise} is meaningful.

Fix an equidistant step size $\delta\in (0,1)$ and a time horizon $T>0$.
To approximate an inverse subordinator $E$ on the interval $[0,T]$, 
we follow the idea presented in \cite{Magdziarz_simulation,Magdziarz_spa}. Namely, we first simulate a sample path of the subordinator $D$, which has independent and stationary increments, by setting  $D_0=0$ and then following the rule $D_{i\delta}:=D_{(i-1)\delta}+Z_i$, $i=1,2,3,\ldots,$
with an i.i.d.\ sequence $\{Z_i\}_{i\in\mathbb{N}}$ distributed as $Z_i=^\mathrm{d} D_{\delta}$. 
We stop this procedure upon finding the integer $N$ satisfying 
	$T\in[D_{N\delta}, D_{(N+1)\delta}).$
 Note that the $\mathbb{N}\cup\{0 \}$-valued random variable $N$ indeed exists since  
 $D_t\to\infty$ as $t\to\infty$ a.s. 
To generate the random variables $\{Z_i\}$, one can use algorithms presented in Chapter  6 of \cite{ContTankov}. 
Next, let
\begin{align}\label{def_Spsidelta}
	E^\delta_t
	:=\bigl(\min\{n\in \mathbb{N}; D_{n\delta}>t\}-1\bigr)\delta, \ \ t\in[0,T].
\end{align}
The sample paths of $E^\delta=(E^\delta_t)_{t\ge 0}$ are nondecreasing step functions with constant jump size $\delta$ and the $i$th waiting time given by $Z_i=D_{i\delta}-D_{(i-1)\delta}$. Indeed, it is easy to see that for $n=0,1,2,\ldots,N$,
\begin{align}\label{property_Spsidelta}
	E^\delta_t=n\delta \ \ \textrm{whenever} \ \ t\in[D_{n\delta},D_{(n+1)\delta}).
\end{align}
In particular, 
$E^\delta_T=N\delta.$
The process $E^\delta$ efficiently approximates $E$;  
indeed, a.s., 
\begin{align}\label{ineq_Spsi}
	E_t-\delta\le E^\delta_t\le E_t \ \ \textrm{for all} \ \  t\in[0,T]. 
\end{align}
For proofs, see \cite{JumKobayashi,Magdziarz_spa}.

Now, let 
\[
	\tau_n=D_{n\delta} \ \ \textrm{for} \ \ n=0,1,2,\ldots,N
\]
and let
\[
	n_t=\max\{n\in \mathbb{N}\cup \{0\}; \tau_n\le t\} \ \ \textrm{for} \ \ t\ge 0.
\]
By the independence assumption between $B$ and $D$, we can approximate the Brownian motion $B$ over the time steps $\{0,\delta,2\delta,\ldots,N\delta\}$, independently of $D$.
Define a discrete-time process $(X^\delta_{\tau_n})_{n\in \{0,1,2,\ldots,N\}}$ by setting 
\begin{align}
	&X^\delta_0
	=x_0\label{Approx_001},\\
	&X^\delta_{\tau_{n+1}}\!\!
	=X^\delta_{\tau_{n}}+F(\tau_n,X^\delta_{\tau_n})\delta+G(\tau_n,X^\delta_{\tau_n})(B_{(n+1)\delta}-B_{n\delta})
	\label{Approx_002}
\end{align}
for $n=0,1,2,\ldots,N-1$.
 Define a continuous-time process $X^\delta=(X^\delta_t)_{t\in[0,T]}$ by piecewise constant interpolation
\begin{align}
	X^\delta_t=X^\delta_{\tau_{n_t}}.\label{Approx_003}
\end{align}
Note that any time point $t\ge 0$ satisfies 
\begin{align}\label{Approx_051}
	\tau_{n_t}\le t<\tau_{n_t+1}
\end{align}
 and that sample paths of $X^\delta$ and $E^\delta$ are both constant over any interval of the form $[\tau_n,\tau_{n+1})$. 
Figure \ref{figure_002} presents a simulation of sample paths of $E$ and $X$ based on this approximation scheme, where the time component of the external force term is taken to be sinusoidal as in e.g.\ \cite{SokolovKlafter2006}.

\begin{figure}
    \centering
    \includegraphics[width=3.8in]{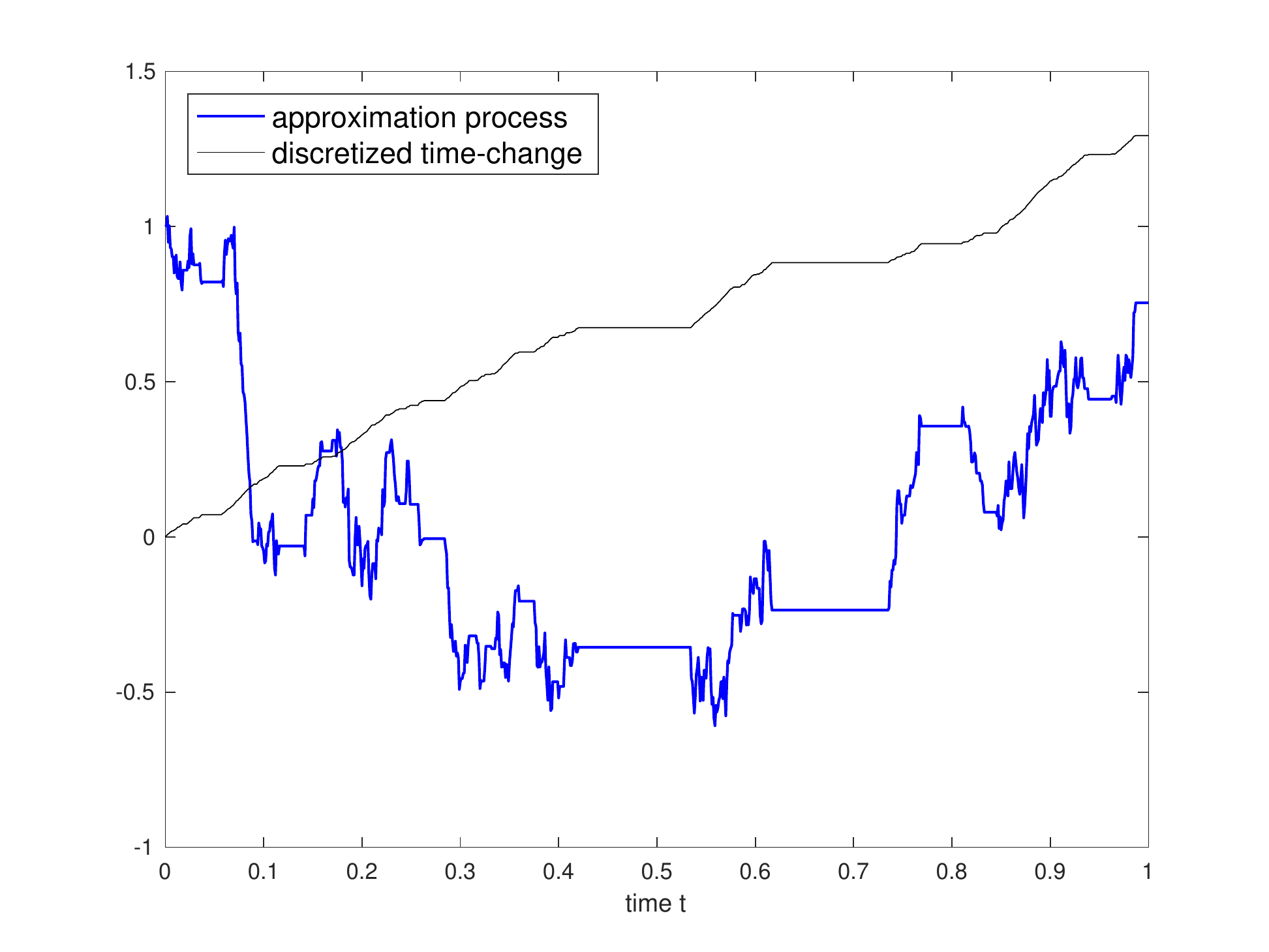}
    \caption{Sample paths of an inverse $0.85$-stable subordinator $E$ (black) and the corresponding solution $X$  (blue) of SDE $X_t=1+\int_0^t (\sin s)X_s\,\dd E_s+B_{E_t}$ on the time interval $[0,1]$.}
    \label{figure_002}
\end{figure}

We now state the main theorem of this paper, which gives the rate of strong convergence of the approximation scheme for SDE \eqref{SDE_001}.
Recall that 
an approximation process $X^\delta$ with step size $\delta>0$ is said to \textit{converge strongly to the solution $X$ uniformly on $[0,T]$ with order $\eta\in(0,\infty)$} 
if there exist finite positive constants $C$ and $\delta_0$ such that 
$
	\mathbb{E}\left[\sup_{0\le t\le T}|X_t-X^\delta_t|\right]\le C\delta^\eta
$
for all $\delta\in(0,\delta_0)$.

\begin{theorem}\label{Theorem_with_noise}
Let $E$ be the inverse of a subordinator $D$ with Laplace exponent $\psi$ and infinite L\'evy measure.
Let $B$ be Brownian motion independent of $D$.
Let $X^\delta$ be the approximation process defined in \eqref{Approx_001}--\eqref{Approx_003} for the exact solution $X$ of SDE \eqref{SDE_001}, where the coefficients $F(t,x)$ and $G(t,x)$ satisfy conditions \eqref{SDE_condition1}--\eqref{SDE_condition4}. 
Suppose further that at least one of the following conditions holds:
\begin{enumerate}
\item[\em{(i)}] $\psi$ is regularly varying at $\infty$ with index $\beta\in(1/2,1)$;
\item[\em{(ii)}] $G(t,x)=G(t)$ for all $(t,x)\in[0,T]\times \R$.
\end{enumerate}
Then for any $\varepsilon\in(0,1/2)$, there exist constants $C\in(0,\infty)$ (not depending on $\delta$) and $\delta_0=\delta_0(\varepsilon)\in(0,1)$ such that
\[
 	\E\left[ \sup_{0\le t\le T}|X_t-X^\delta_t|\right]\le C(\delta^{\theta_F}+\delta^{\theta_G}+\delta^{1/2-\varepsilon})
\]
 for all $\delta\in(0,\delta_0)$. Thus, $X^\delta$ converges strongly to $X$ uniformly on  $[0,T]$ with order 
$\min({\theta_F},\theta_G,1/2-\varepsilon)$.
\end{theorem}

The proof of Theorem \ref{Theorem_with_noise} is based on the following lemma.

\begin{lemma}\label{lemma_Holder}
Let $E$ be the inverse of a subordinator $D$ with infinite L\'evy measure. 
Let $B$ be Brownian motion independent of $D$.
Let $X$ be the solution of SDE \eqref{SDE_001}, where the coefficients $F(t,x)$ and $G(t,x)$ satisfy conditions \eqref{SDE_condition1}--\eqref{SDE_condition2}. 
For a fixed $p\in[1,\infty)$, let $Y^{(p)}_T:=1+\sup_{0\le s\le T} |X_s|^p$.     
Then $\E[Y^{(p)}_T]<\infty$.
\end{lemma}

To prove the lemma, 
let us recall two inequalities.
First, for any $p\in[1,\infty)$, the inequality 
\[
	(x+y+z)^p\le c_p (x^p+y^p+z^p)
\]
is valid for all $x,y,z\ge 0$, 
where $c_p=3^{p-1}$. 
Second, the Burkholder--Davis--Gundy inequality states that for any $p>0$, there exists a constant $b_p>0$ such that 
\begin{align}\label{Burkholder}
	\E\left[ \sup_{0\le t\le S} |M_t|^p\right]\le b_p \E\left[ [M,M]_S^{p/2}\right]
\end{align}
for any stopping time $S$ and any continuous local martingale $M$ with quadratic variation $[M,M]$. The constant $b_p$ can be taken independently of $S$ and $M$; see Proposition 3.26 and Theorem 3.28 of Chapter 3 of \cite{KaratzasShreve}.

Since the Brownian motion $B$ and the subordinator $D$ are assumed independent, it is possible to set up $B$ and $D$ on a product space with product measure $\P=\P_B\times \P_D$ with obvious notation. We use this set up in the proofs of Lemma \ref{lemma_Holder} and Theorem \ref{Theorem_with_noise} below. Let $\E_B$, $\E_D$ and $\E$ denote the expectations under the probability measures $\P_B$, $\P_D$ and $\P$, respectively.

\begin{proof}[Proof of Lemma \ref{lemma_Holder}]
It suffices to prove the statement for $p\ge 2$ since the result for $1\le p<2$ follows immediately from the result for $p\ge 2$ with Jensen's inequality. 
Fix $p\ge 2$ and let $Y^{(p)}_t:=1+\sup_{0\le s\le t} |X_s|^p$ for $t\in[0,T]$. 
 Let $S_\ell:=\inf \{t\ge 0; Y^{(p)}_t>\ell\}$ for $\ell\in\mathbb{N}$. 
Since the solution $X$ has continuous paths, $Y^{(p)}_t<\infty$, and hence, $S_\ell\uparrow \infty$ as $\ell\to \infty$.
The idea of the proof is to first apply a Gronwall-type inequality to
the function $t\mapsto\E_B[Y^{(p)}_{t\wedge S_\ell}]$ for a fixed $\ell$ and then let $t=T$ and $\ell\to\infty$ in the obtained inequality to establish a desired bound for $\E_B[Y^{(p)}_T]$.  Note that we introduced the localizing sequence $\{S_\ell; \ell\in\mathbb{N}\}$ 
in order to guarantee that 
$\int_0^t \E_B[Y^{(p)}_{r\wedge S_\ell}]\,\dd E_r\le \ell E_t<\infty$, 
which enables us to apply the Gronwall-type inequality. 

Fix $\ell\in\mathbb{N}$ and $t\in[0,T]$. Since $E$ has continuous paths, by Theorem 10.17 of \cite{Jacod} (also see Lemma 2.4 and Example 2.5 of \cite{Kobayashi}), $B\circ E$ is a continuous martingale with quadratic variation
$
	[B\circ E,B\circ E]=[B,B]\circ E=E.
$
By the It\^o formula and representation \eqref{SDE_001}, 
\[
	X^p_s
	=x_0^p+\int_0^s \left\{pX_r^{p-1} F(r,X_r)+\frac 12p(p-1)X_r^{p-2} G^2(r,X_r)\right\}\,\dd E_r+M_s,
\]
where
\[
	M_s=\int_0^s pX_r^{p-1} G(r,X_r)\,\dd B_{E_r}.
\]
 Due to condition \eqref{SDE_condition2}, the absolute value of the integrand of the $\dd E_r$ integral is dominated by 
\begin{align*}
	pK|X_r|^{p-1}(1+|X_r|)+\frac 12 p(p-1)K^2|X_r|^{p-2}(1+|X_r|)^2
	\le A_1Y_r^{(p)}, 
\end{align*}
where $A_1=pc_pK+p(p-1)c_pK^2/2$. Thus, 
\begin{align}\label{R01}
	Y^{(p)}_{t\wedge S_\ell}
	=1+\sup_{0\le s\le t\wedge S_\ell} |X_s|^p
	\le 1+|x_0|^p+I_1+I_2,
\end{align}
where 
\[
	I_1= A_1\int_0^{t\wedge S_\ell} Y^{(p)}_r\,\dd E_r \ \ \ \textrm{and} \ \ \ 
	I_2=\sup_{0\le s\le {t\wedge S_\ell}}|M_s|.
\]
Note that for any nonnegative process $u(r)$, the inequality 
\begin{align}\label{Holder_072}
	\int_0^{t\wedge S_\ell} u(r)\,\dd E_r\le \int_0^t u({r\wedge S_\ell})\,\dd E_r
\end{align}
holds.
Indeed, the inequality obviously holds if $t\le S_\ell$, while if $t>  S_\ell$, then 
$
	\int_0^t u({r\wedge S_\ell})\,\dd E_r
	=\int_0^{S_\ell} u(r)\,\dd E_r+\int_{S_\ell}^t u({S_\ell})\,\dd E_r
	\ge \int_0^{t\wedge S_\ell} u(r)\,\dd E_r,
$
thereby yielding \eqref{Holder_072}. 
Thus, 
\begin{align}\label{R02}
	\E_B[I_1]\le A_1\int_0^t \E_B[Y^{(p)}_{r\wedge S_\ell}]\,\dd E_r.
\end{align}

To deal with $I_2$, note that the stochastic integral $(M_t)_{t\ge 0}$ is a local martingale
 with quadratic variation $[M,M]_t=\int_0^t p^2X_r^{2p-2}G^2(r,X_r) \,\dd E_r$ since stochastic integration preserves the local martingale property; see Chapter III, Theorem 29 in \cite{Protter}.
By condition \eqref{SDE_condition2}, for $0\le r\le t\wedge S_\ell$,
\begin{align*}
	p^2X_r^{2p-2}G^2(r,X_r)
	\le p^2 K^2X_r^{2p-2}(1+|X_r|)^2
	\le p^2c_p^2K^2  Y^{(p)}_{t\wedge S_\ell} Y^{(p)}_r,
\end{align*}
and hence, $([M,M]_{t\wedge S_\ell})^{1/2}$ is dominated by
\begin{align*}
	p c_p K \left(Y^{(p)}_{t\wedge S_\ell}\int_0^{t\wedge S_\ell} \hspace{-4mm}Y^{(p)}_r\,\dd E_r\right)^{1/2}
	\le p c_p K\left( \frac{Y^{(p)}_{t\wedge S_\ell}}{2b_1pc_pK}+2b_1pc_pK \int_0^{t\wedge S_\ell}  \hspace{-4mm}Y^{(p)}_r\,\dd E_r  \right),
\end{align*}
where we used the inequality $(ab)^{1/2}\le a/\lambda +\lambda b$ valid for any $a,b\ge 0$ and $\lambda>0$, and $b_1$ is the constant appearing in the Burkholder--Davis--Gundy inequality \eqref{Burkholder}. Applying the latter inequality and inequality \eqref{Holder_072} now gives 
\begin{align}\label{R03}
	\E_B[I_2]
	\le b_1\E_B\left[ ( [M,M]_{t\wedge S_\ell} )^{1/2}\right]
	\le \frac 12 \E_B[ Y^{(p)}_{t\wedge S_\ell}] +A_2 \int_0^t \E_B[Y^{(p)}_{r\wedge S_\ell}]\, \dd E_r,
\end{align}
where $A_2=2b_1 p^2 c_p^2K^2$. 
Note that the constant $b_1$ is independent of the stopping time ${t\wedge S_\ell}$, and in particular, $A_2$ does not depend on $\ell$.  
Taking $\E_B$ 
on both sides of \eqref{R01} and using \eqref{R02} and \eqref{R03} yields
\[
	\E_B[Y^{(p)}_{t\wedge S_\ell}]
	\le 1+|x_0|^p+\frac 12 \E_B[ Y^{(p)}_{t\wedge S_\ell}] +(A_1+A_2) \int_0^t \E_B[Y^{(p)}_{r\wedge S_\ell}]\, \dd E_r,
\]
which in turn gives  
$
	\E_B[Y^{(p)}_{t\wedge S_\ell}]\le 2(1+|x_0|^p)+2(A_1+A_2)\int_0^t \E_B[Y^{(p)}_{r\wedge S_\ell}]\, \dd E_r. 
$
Thus, applying a Gronwall-type inequality in Chapter IX.6a, Lemma 6.3 of \cite{JacodShiryaev} yields
\begin{align*}
	\E_B[Y^{(p)}_{t\wedge S_\ell}]
	\le 2(1+|x_0|^p) e^{2(A_1+A_2) E_t}
\end{align*}
for all $t\in[0,T]$. Note that $E_t$ appears in the exponent on the right hand side since the integral above is driven by the process $E$.
Setting $t=T$, letting $\ell\to\infty$ while recalling $A_1$ and $A_2$ do not depend on $\ell$, and using the monotone convergence theorem yields $\E_B[Y^{(p)}_T]\le 2(1+|x_0|^p) e^{2(A_1+A_2) E_T}$. Taking $\E_D$ on both sides, noting $\E_D[\E_B[Y^{(p)}_T]]=\E[Y^{(p)}_T]$, and using the fact that $\E[e^{\lambda E_T}]<\infty$ for any $\lambda>0$ yields 
the desired result.
\end{proof}

\begin{proof}[Proof of Theorem \ref{Theorem_with_noise}]
Let
\[
	Z_t:=\sup_{0\le s\le t}|X_s-X^\delta_s| \ \ \textrm{for} \ \ t\in[0,T].
\] 
As in the proof of Lemma \ref{lemma_Holder}, we use the localizing sequence $S_\ell=\inf\{t\ge 0;Z_{t+1}>\ell\}$, 
which allows us to safely apply a Gronwall-type inequality.  
Here, $S_\ell$ is defined as $\inf\{t\ge 0; Z_{t+1}>\ell\}$ instead of $\inf\{t\ge 0; Z_t>\ell\}$ in order to guarantee that $Z_{t\wedge S_\ell}\le \ell$ even if the process $Z$ exceeds the level $\ell$ by a jump. Note that the approximation $X^\delta$ is a piecewise constant process with finitely many jumps, so $Z_t<\infty$, and hence, $S_\ell\uparrow \infty$ as $\ell\to\infty$.
  In the remainder of the proof, however, to clarify the main ideas, we assume the function $t\mapsto Z_t$ is bounded.

By \eqref{Approx_001}--\eqref{Approx_003},
\begin{align*}
	X^\delta_s-x_0
	=\sum_{i=0}^{n_s-1}(X^\delta_{\tau_{i+1}}\hspace{-2pt}-X^\delta_{\tau_i})
	=\sum_{i=0}^{n_s-1} \left(F(\tau_i,X^\delta_{\tau_i})\delta +  G(\tau_i, X^\delta_{\tau_i})(B_{(i+1)\delta}-B_{i\delta})\right). 
\end{align*}
Note that $E_{\tau_i}=E_{D_{i\delta}}=i\delta$ and that $\tau_i=\tau_{n_r}$ for any $r\in[\tau_i,\tau_{i+1})$, which implies the above can be rewritten as
\[
	X^\delta_s-x_0
	=\int_0^{\tau_{n_s}} F(\tau_{n_r},X^\delta_r)\,\dd E_r
		+\int_0^{\tau_{n_s}} G(\tau_{n_r},X^\delta_r)\,\dd B_{E_r}.
\]
Hence, for a fixed $t\in[0,T]$,
\begin{align}\label{Eq021}
	Z_t\le I_1+I_2+I_3+I_4, \ \ 
	Z_t^2\le 4(I_1^2+I_2^2+I_3^2+I_4^2),
\end{align}
where
\begin{align*}
	I_1&=\hspace{-1pt}\sup_{0\le s\le t} \left| \int_0^{\tau_{n_s}} \hspace{-2mm}(F(r,X_r)\hspace{-1pt}-\hspace{-1pt}F(\tau_{n_r},X^\delta_r))\,\dd E_r \right|,
	\hspace{3mm}I_2=\hspace{-1pt}\sup_{0\le s\le t} \left| \int_{\tau_{n_s}}^s \hspace{-2mm}F(r,X_r)\,\dd E_r\right |,\\
	I_3&=\hspace{-1pt}\sup_{0\le s\le t} \left| \int_0^{\tau_{n_s}} \hspace{-2mm}(G(r,X_r)\hspace{-1pt}-\hspace{-1pt}G(\tau_{n_r},X^\delta_r))\,\dd B_{E_r} \right|, \
	I_4=\hspace{-1pt}\sup_{0\le s\le t} \left| \int_{\tau_{n_s}}^s \hspace{-2mm}G(r,X_r)\,\dd B_{E_r} \right|. 
\end{align*}

In terms of $I_1$, note that for $r\in[0,\tau_{n_t})$, by conditions \eqref{SDE_condition1} and \eqref{SDE_condition3},
\begin{align}\label{R11}
\left|F(r,X_r)-F(\tau_{n_r},X^\delta_r)\right|
&\le \left|F(r,X_r)-F(\tau_{n_r},X_r)\right|+\left|F(\tau_{n_r},X_r)-F(\tau_{n_r},X^\delta_r)\right|\notag\\
&\le K(1+|X_r|)(r-\tau_{n_r})^{\theta_F} +K|X_r-X^\delta_r|\notag\\
&\le KY^{(1)}_{\tau_{n_t}}(\tau_{n_r+1}-\tau_{n_r})^{\theta_F} +KZ_r,\hspace{-10mm}
\end{align}
where $Y^{(p)}_t=1+\sup_{0\le s\le t} |X_s|^p$ as in the proof of Lemma \ref{lemma_Holder}. 
The function $g(x):=x^{1/{\theta_F}} \ (x\ge 0)$ is convex since ${\theta_F}\in(0,1]$, so by Jensen's inequality, 
 \begin{align*}
	&\left(\sum_{i=0}^{n_t-1}(\tau_{i+1}-\tau_i)^{\theta_F}\right)^{1/{\theta_F}}
	=g\left(\frac{1}{n_t}\sum_{i=0}^{n_t-1} n_t (\tau_{i+1}-\tau_i)^{\theta_F}\right)\\
	&\le \frac {1}{n_t} \sum_{i=0}^{n_t-1}g\left(n_t (\tau_{i+1}-\tau_i)^{\theta_F}\right)
	=n_t^{1/{\theta_F}-1}\sum_{i=0}^{n_t-1} (\tau_{i+1}-\tau_i)
	=n_t^{1/{\theta_F}-1}\tau_{n_t}. 
 \end{align*}
This, together with the identity $\int_{\tau_i}^{\tau_{i+1}}(\tau_{n_r+1}-\tau_{n_r})^{\theta_F}\,\dd E_r=\delta (\tau_{i+1}-\tau_i)^{\theta_F}$, yields
 \begin{align}\label{Eq032}
 	 \int_0^{\tau_{n_t}}(\tau_{n_r+1}-\tau_{n_r})^{\theta_F}\,\dd E_r
	= \delta \sum_{i=0}^{n_t-1}(\tau_{i+1}-\tau_i)^{\theta_F} 
	\le \delta n_t^{1-{\theta_F}}\tau_{n_t}^{\theta_F}.
 \end{align}
By \eqref{property_Spsidelta}--\eqref{ineq_Spsi}, $n_t$ is a random variable satisfying the relation $n_t\delta =E^\delta_t\le E_t$. Using the inequalities $\tau_{n_t}\le t\le T$ 
 and putting together \eqref{R11} and \eqref{Eq032} yields
 \begin{align}\label{Eq002}
I_1\le K \delta^{\theta_F} T^{\theta_F}  E_T^{1-{\theta_F}}Y^{(1)}_T
+K\int_0^t Z_r\,\dd E_r. 
\end{align}
By the Cauchy--Schwarz inequality, this implies
\begin{align}\label{Eq022}
	I_1^2
	&\le 4K^2 \delta^{2{\theta_F}} T^{2{\theta_F}}  E_T^{2(1-{\theta_F})}Y^{(2)}_T 
+2K^2 E_T \int_0^t Z_r^2\,\dd E_r.
\end{align}

On the other hand, in terms of $I_2$, by condition \eqref{SDE_condition2},
\begin{align}\label{Eq003}
	I_2
	\le \sup_{0\le s\le t} K\int_{\tau_{n_s}}^{s}  (1+|X_r|)\,\dd E_r 
	\le \sup_{0\le s\le t} K Y^{(1)}_s (E_s-E_{\tau_{n_s}})
	\le K\delta Y^{(1)}_T,
\end{align}
which implies 
\begin{align}\label{Eq023}
	I_2^2
	\le 2K^2\delta^2 Y^{(2)}_T.
\end{align}
To deal with $I_3$, note that $B\circ E$ is a martingale with quadratic variation $E$, and hence, 
by the Burkholder--Davis--Gundy inequality \eqref{Burkholder}, 
\begin{align*}
	\E_B[I_3^2]
	&\le \E_B\left[\sup_{0\le s\le \tau_{n_t}} \left|\int_0^{s} (G(r,X_r)-G(\tau_{n_r},X^\delta_r)) \,\dd B_{E_r}\right|^2\right]\notag\\
	&\le b_2 \E_B\left[\int_0^{\tau_{n_t}} (G(r,X_r)-G(\tau_{n_r},X^\delta_r))^2 \,\dd E_r\right].
\end{align*}
 An estimation similar to the one in \eqref{R11} gives
\begin{align*}
	\E_B[I_3^2]
	\le 4K^2b_2 \E_B [Y^{(2)}_{\tau_{n_t}}]
	\int_0^{\tau_{n_t}}(\tau_{n_r+1}-\tau_{n_r})^{2{\theta_G}}\,\dd E_r
	+2K^2b_2 \int_0^{\tau_{n_t}} \E_B[Z_r^2]\,\dd E_r. 
 \end{align*}
Note that for ${\theta_G}\in (1/2,1]$, we have $|t-s|^{\theta_G}\le (2T)^{{\theta_G}-1/2}|t-s|^{1/2}$, so there is no loss of generality in assuming ${\theta_G}\in(0,1/2]$ in \eqref{SDE_condition4}.
Assuming ${\theta_G}\in(0,1/2]$, we can use \eqref{Eq032} with ${\theta_F}\in(0,1]$ replaced by $2{\theta_G}$, yielding 
\begin{align}\label{Eq034}
	\E_B[I_3^2]
	&\le 4K^2b_2 \E_B[Y^{(2)}_{\tau_{n_t}}] \delta n_t^{1-2{\theta_G}}\tau_{n_t}^{2{\theta_G}}
	+2K^2b_2 \int_0^{\tau_{n_t}} \E_B[Z_r^2]\,\dd E_r\notag\\
	&\le 4\delta^{2{\theta_G}}K^2b_2 T^{2{\theta_G}}E_T^{1-2{\theta_G}}  \E_B[Y^{(2)}_T]
	+2K^2b_2 \int_0^t \E_B[Z_r^2]\,\dd E_r, 
\end{align}
where we used $n_t\delta=E^\delta_t\le E_t$ and $\tau_{n_t}\le t\le T$.  

Estimation of $I_4^2$ requires some careful work.
By the change-of-variable formula for stochastic integrals driven by time-changed semimartingales (Theorem 3.1 of \cite{Kobayashi}) and the relation $E_{\tau_{n_s}}=n_s\delta=E^\delta_s$, it follows that 
\[
	\E_B[I_4^2]
	=\E_B\left[ \sup_{0\le s\le t} \left| \int_{E^\delta_s}^{E_s} \hspace{-2mm}G(D_{v-},X_{D_{v-}})\,\dd B_v \right|^2\right]
	\le\E_B\left[\sup_{0\le r,s\le E_T,\,0\le s-r\le \delta} \hspace{-2mm}N_{r,s}^2\right],
\]
where $N_{r,s}=\int_r^s G(D_{v-},X_{D_{v-}})\,\dd B_v$.
Note that the use of the change-of-variable formula requires that the semimartingale ($B$ in this case) be constant on any interval of the form $[E_{t-},E_t]$, but this is satisfied since the time change $E$ has continuous paths. 
Observe that 
\begin{align}
	&\E_B[I_4^2]
	=\E_B[\mathbf{1}_{E_T\ge\delta}\cdot I_4^2]+\E_B[\mathbf{1}_{E_T<\delta}\cdot I_4^2]\notag\\
	&\le\E_B\left[\mathbf{1}_{E_T\ge\delta}\cdot \sup_{0\le r,s\le E_T,\,0\le s-r\le \delta} N_{r,s}^2\right]
	+\E_B\left[\sup_{0\le r,s\le \delta,\,0\le s-r\le \delta} N_{r,s}^2\right],\label{R21}
\end{align}
where $\mathbf{1}_A$ denotes the indicator function on a set $A$. 
Also, for all $0\le r<s\le u$,
\[
	\int_r^s G^2(D_{v-},X_{D_{v-}})\, \dd v \le K^2 \int_r^s (1+|X_{D_{v-}}|)^2\, \dd v\le \xi(u)|s-r|,
\]
where 
$
	\xi(u):=2K^2 Y^{(2)}_{D_{u-}}.
$
By Lemma \ref{lemma_Holder}, 
for $\alpha>0$, 
\[
	\E_B[\xi(E_T)^{1+\alpha}]
	\le (2K^2)^{1+\alpha} c_{1+\alpha} 
	\E_B[Y^{(2(1+\alpha))}_T]
	<\infty
\]
and $\E_B[\xi(\delta)^{1+\alpha}]<\infty$ $\P_D$-a.s.
These guarantee that Theorem 1 of \cite{FischerNappo}, which concerns modulus of continuity for stochastic integrals, is applicable. Namely, there exists a constant $C_1$ such that 
$
	\E_B\left[\sup_{0\le r,s\le u, \, 0\le s-r\le \delta} N_{r,s}^2\right]
	\le C_1 \delta \log\left(2 u/\delta\right)
$
for all $0<\delta\le u$ with $u=E_T$ and $u=\delta$, and their proof shows that $C_1$ can be taken independently of $u$. This, together with \eqref{R21}, gives
\begin{align}
	\E_B[I_4^2]
	\le C_1 \delta \log\left(\frac{2 E_T}{\delta}\right)+C_1 \delta \log 2
	=C_1 \delta \log\left(\frac {4 E_T}{\delta}\right).\label{R13}
\end{align}

Now, let us assume that condition (i) holds; i.e. $\psi\in \mathrm{RV}_\beta$ with $\beta\in(1/2,1)$.
 Putting together the estimates \eqref{Eq021}, \eqref{Eq022}, \eqref{Eq023}, \eqref{Eq034} and \eqref{R13} gives
\begin{align}
	\E_B[Z_t^2]
    \le (V_1+V_2)+8K^2(E_T+b_2)\int_0^t \E_B[Z_r^2]\,\dd E_r,\label{R05}
\end{align}
where 
\[
	V_1=4C_1 \delta \log\left(\frac {4 E_T}{\delta}\right), \ \ 
	V_2=C_2 (\delta^2+E_T^{2(1-{\theta_F})}\delta^{2{\theta_F}}+E_T^{1-2{\theta_G}}\delta^{2{\theta_G}})\E_B[Y^{(2)}_T]
\]
with $C_2$ being a constant depending on $K$, $T$, ${\theta_F}$ and ${\theta_G}$.
Applying a Gronwall-type inequality in Chapter IX.6a, Lemma 6.3 of \cite{JacodShiryaev} and taking $\E_D$
on both sides of the obtained inequality with $t=T$ gives
\begin{align*}
	\E[Z_T^2]\le \E[(V_1+V_2)e^{8K^2(E_T+b_2)E_T}]
	\le \left(\E[(V_1+V_2)^2]\cdot \E[e^{16K^2(E_T+b_2)E_T}]\right)^{1/2}.
\end{align*}
The assumption $\beta\in(1/2,1)$ allows us to use Theorem \ref{Theorem_general} with $r=2<1/(1-\beta)$, which implies $\E[e^{16K^2(E_T+b_2)E_T}]<\infty$. Moreover, $\E[V_2^2]<\infty$ since $\E[Y^{(p)}_T]<\infty$ for any $p\ge 1$ by Lemma \ref{lemma_Holder}. Let $\ve\in(0,1/2)$. 
Note that $\log x<\ve^{-1} x^{2\ve}$ for all $x>0$, so $(\delta^{2\varepsilon} \log(4E_T/\delta))^{2}\le \ve^{-2}(4E_T)^{4\varepsilon}$.
As the right hand side has finite expectation and is independent of $\delta$, by the dominated convergence theorem,
$
	\delta^{2(2\varepsilon-1)}
		\E[V_1^{2}]
	=16C_1^2\E[(\delta^{2\ve} \log(4 E_T/\delta))^{2}] \to 0
$
as $\delta\downarrow 0$. Therefore, there exists $\delta_0=\delta_0(\varepsilon)\in (0,1)$ such that 
$\E[V_1^{2}]<\delta^{2(1-2\varepsilon)}$ for all $\delta\in(0,\delta_0)$. We have thus obtained the inequality 
$\E[Z_T^2]\le C_3(\delta^{1-2\ve}+\delta^{2\theta_F}+\delta^{2\theta_G})$ for some finite constant $C_3$. The obvious inequality $\E[Z_T]\le (\E[Z_T^2])^{1/2}$ now completes the proof of the theorem with condition (i).

We now turn to the proof of the theorem 
with condition (ii) that $G(t,x)=G(t)$ for $(t,x)\in[0,T]\times \R$. Since $G(t,x)$ does not depend on $x$, the integral $\int_0^t \E_B[Z_r^2]\,\dd E_r$  in \eqref{Eq034} vanishes. Moreover, since condition \eqref{SDE_condition4} 
simplifies to
$|G(s)-G(t)|\le K|s-t|^{\theta_G}$, the expression $ \E_B[Y^{(2)}_T]$ in \eqref{Eq034} also disappears.
Thus, by inequalities \eqref{Eq021}, \eqref{Eq002}, \eqref{Eq003}, \eqref{Eq034} and \eqref{R13}, as well as the obvious inequality $\delta\le \delta^{\theta_F}$ valid for $\delta\in(0,1)$,
\begin{align}\label{R15}
	\E_B[Z_t]
	\le V_3+K\int_0^{t} \E_B[Z_r]\,\dd E_r, 
\end{align}
where 
\[
	V_3=K\delta^{\theta_F} (1+  T^{\theta_F} E_T^{1-{\theta_F}})\E_B[Y^{(1)}_T]+C_4 \delta^{\theta_G}E_T^{1/2-{\theta_G}}+\left(C_1\delta \log\left(\frac{4E_T}{\delta}\right)\right)^{1/2}
\]
with $C_4$ being a constant depending on $K$, $T$, $\theta_G$ and $b_2$.
Applying a Gronwall-type inequality and 
taking $\E_D$ on both sides of the obtained inequality with $t=T$ gives 
\[
	\E[Z_T]
	\le \E[V_3e^{KE_T}]
	\le \left(\E[V_3^2]\cdot \E[e^{2KE_T}]\right)^{1/2}. 
\]
Note that since $\E[e^{2KE_T}]<\infty$ for any inverse subordinator with underlying L\'evy measure being infinite, we do not need to impose condition (i) that $\psi\in \mathrm{RV}_\beta$ with $\beta\in(1/2,1)$. 
The remainder of the proof is omitted since it is very similar to the argument given in the last part of the previous paragraph.
\end{proof}

\begin{remark}\label{Remark_with_noise}
\begin{em}
(1) Theorem \ref{Theorem_with_noise} is still valid when the underlying Laplace exponent takes the form $\widehat{\psi}(s)=as+\psi(s)$ with $a>0$ and $\psi\in \mathrm{RV}_\beta$ as in Remark \ref{Remark_general}(1). 

(2) If $G(t,x)$ depends on $x$, then in the above proof
the analysis of the squared error $Z_t^2$ instead of the $p$th power error $Z_t^p$
with $p\ne 2$ is essential.
Indeed, a straightforward modification of the proof would not lead to an inequality for $\E_B[Z^{p}_t]$ to which the Gronwall-type inequality is readily applicable. This is because replacing the integral $\int_0^t \E_B[Z_r^2]\,\dd E_r$ in \eqref{Eq034} by the integral $\int_0^t \E_B[Z^{p}_r]\,\dd E_r$ using the current method does not seem to be possible. Moreover, the appearance of the integral $\int_0^t \E_B[Z_r^2]\,\dd E_r$ in \eqref{Eq034} requires the estimate \eqref{Eq022} for $I_1^2$ be used instead of the estimate \eqref{Eq002} for $I_1$. The presence of $E_T$ in front of the integral $\int_0^t Z_r^2\,\dd E_r$ in \eqref{Eq022} (and hence in \eqref{R05} as well) is what amounts to the expression
$\E[e^{16K^2(E_T+b_2)E_T}]$. 
Condition (i) was imposed to guarantee the finiteness of the latter.

(3) The Euler--Maruyama scheme for classical It\^o SDEs (without a random time change) has order 1/2 of strong uniform convergence (see Theorem 10.2.2 of \cite{KloedenPlaten}). On the other hand, for SDE \eqref{SDE_001}, it only seems possible to derive an order strictly less than $1/2$. This is because even in the simple case when $G(t,x)\equiv 1$, in order to control the quantity $I_4=\sup_{0\le s\le t}|B_{E_s}-B_{E^\delta_s}|$, we need to use a result about the modulus of continuity for Brownian motion, which involves a logarithmic correction. 
However, the rate of convergence at the time horizon $T$ can be slightly improved since the discussion of the modulus of continuity is unnecessary.
  Namely, it follows that
 $
 	\E\left[ |X_T-X^\delta_T|\right]\le C(\delta^{{\theta_F}}+\delta^{{\theta_G}}+\delta^{1/2})
 $
 for all $\delta\in(0,1)$. 
\end{em}
\end{remark}

\noindent
\textbf{Acknowledgements:} 
The authors appreciate various comments and suggestions by anonymous referees that resulted in a succinct, better-organized paper with an improved result.

\includepdf[]{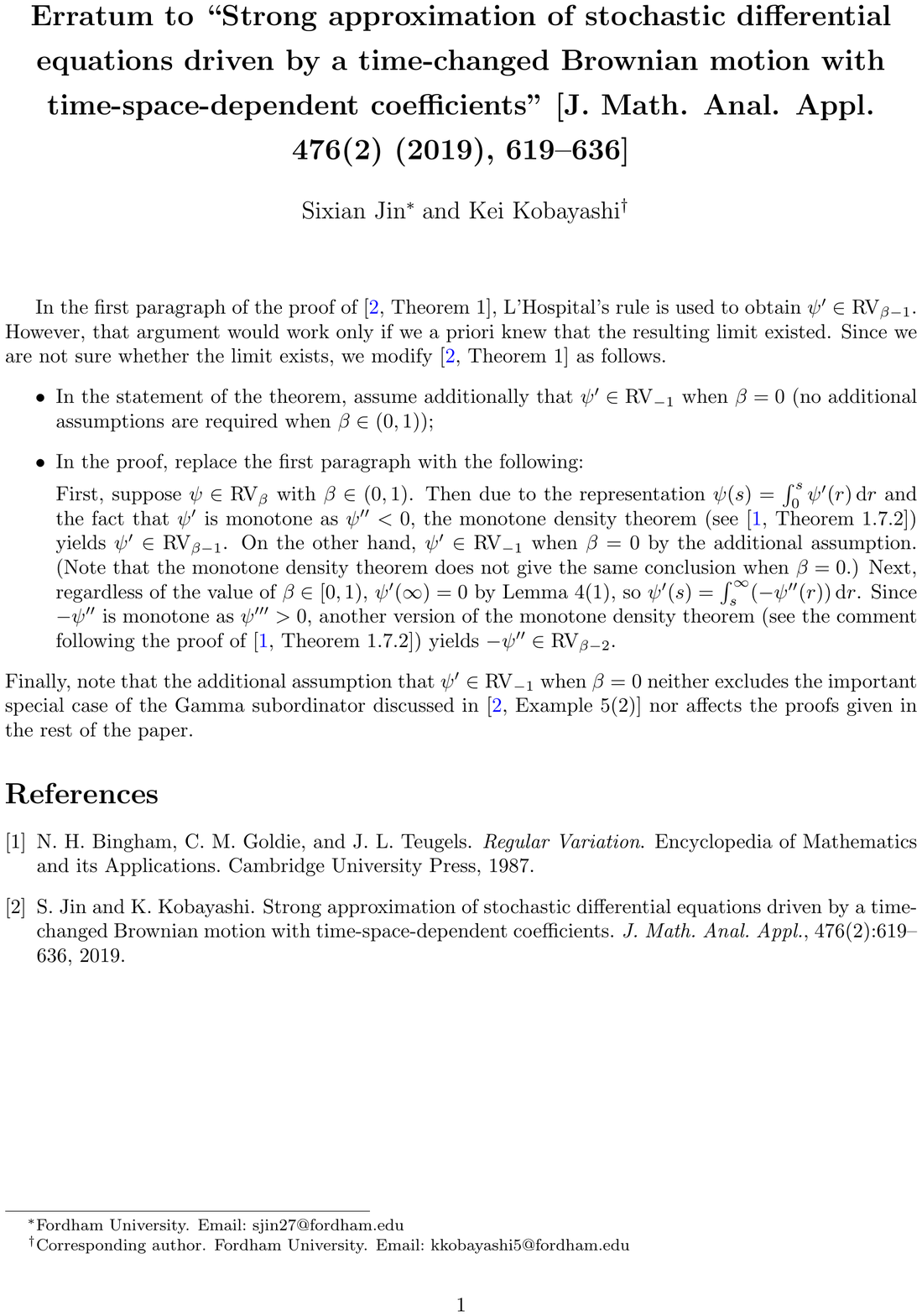}
\end{document}